\documentclass[12pt]{article}
\usepackage{amsthm}
\usepackage{amsfonts}
\usepackage{amssymb}
\usepackage{amsmath}
\usepackage{nccmath}
\usepackage{mathtools}
\usepackage{mathrsfs}
\usepackage{setspace}
\usepackage{graphicx}
\usepackage{hyperref}
\usepackage[numbers,sort&compress]{natbib}
\usepackage{enumerate}
\usepackage{authblk}

\newtheorem{theorem}{Theorem}[section]
\newtheorem{lemma}[theorem]{Lemma}

\theoremstyle{definition}

\theoremstyle{remark}

\numberwithin{equation}{section}

\usepackage{amsmath,amssymb,graphicx} 
\usepackage[margin=1in]{geometry} 
\usepackage{enumerate}
\usepackage{authblk}
\usepackage{placeins}
\usepackage{array}
\usepackage{academicons}
\usepackage{xcolor}
\usepackage{svg}
\usepackage[utf8]{inputenc}
\providecommand{\keywords}[1]{\textbf{\textit{Keywords:}} #1}
\providecommand{\subjclass}[1]{\textbf{\textit{MSC2020:}} #1}
\begin{document}

\nocite{*} 

\title{Closed-form formula for some recursively-defined integro-difference sequence of functions}

\author{Hailu Bikila Yadeta \\ email: \href{mailto:haybik@gmail.com}{haybik@gmail.com} }
  \affil{Dılla University, College of  Natural and Computational Sciences, Department of Mathematics, Dilla, Ethiopia}
\date{\today}
\maketitle
\noindent

\begin{abstract}
The main purpose of this  paper is to  derive  the closed form solution the sequence $(g_n)_{n\in \mathbb{N}}$ of integro-difference equations that is defined recursively as follows:
\begin{align*}
  g_1(x) & = \chi_{(-1/2, 1/2)} (x), \\
  g_{n+1}(x) & = g_n(x + 1/2)- g_n(x- 1/2) + \int_{x-\frac{1}{2}}^{x + \frac{1}{2}}  g_n(s)ds, \, n\in \mathbb{N},
\end{align*}
where $ g_1(x)= \chi_{(-1/2, 1/2)} (x) $ is the characteristic function of the unit interval $(-1/2, 1/2) $ has value equal to $ 1 $ on $(-1/2, 1/2) $ and $0$ elsewhere in $ \mathbb{R} $.
\end{abstract}

\noindent\keywords{binomial coefficient,  integro-difference equation, recursively-defined sequences, shift operator}\\
\subjclass{Primary 45E10, 39B12,39B22}\\
\subjclass{Secondary 26A18}

\section{Introduction}
 Integro difference equations are applicable in the modeling of certain real life problems. For example, in \cite{KLN} deterministic integro difference equations for a single, unstructured population living on a one-dimensional spatial habitat is typically written as
 \begin{equation}\label{eq:stateequation}
    n_{t+1} = \int_{\Omega} k(x,y)f(n_t(y))dy.
 \end{equation}
The state variable $n_t(x) $ is the population density at a location $x$ and a time $t$. Space is continuous while time is discrete. Thus $ x $ and $y$ are real numbers but $t$ is an integer $ \Omega $ is a spacial domain for the population and the model. The domain may be finite like the interval $(0,1) $. For some problems, however, it is convenient to take the domain and the limits of integration infinite. $\Omega =( -\infty, \infty) $. The dispersal kernel $k(x,y)$, can depend upon the source $y$ and the destination $x$ in some complicated way. To simplify matters, ecologists often assume a difference kernel,
$$   k(x,y)=k(x-y)$$
so that  the integro-difference equation  becomes convolution type
 $$   n_{t+1} = \int_{\Omega} k(x,y)f(n_t(y))dy.  $$
Such a convolution type integro-difference equation was used  for approximating the speed of spread of a population is presented in (\cite{FL}) as
\begin{equation}\label{eq:convolutiontype}
  N_{t+1}= R\int_{-\infty}^{\infty}k(x-y)N_t(y)dy= R( K \ast N_t)(x).
\end{equation}
With the localized initial condition $N_0(x)=\delta(x)$, the solution for (\ref{eq:convolutiontype})  is given by
$$  N_t(x)=  R^t K^{\ast t}(x), $$
where $K^{\ast t}(x) := \underbrace{ K \ast K \ast...\ast K}_{\text{t times}} $. The calculation of such a $t$-fold convolution  for arbitrary $K$ may  not be simple. In connection to some previous references and applications in other areas of study, the current author calculated, in \cite{HBY},  the closed form  expression of the recursively defined sequences of functions defined by recursive convolution

\begin{equation}\label{eq:integralrecursion}
  f_1(x)  = \chi_{(-1/2, 1/2)} (x),\quad f_{n+1}(x)  = \int_{x-\frac{1}{2}}^{x+\frac{1}{2}}  f_n(s)ds,  \, n\in \mathbb{N},
\end{equation}
was calculated as
\begin{equation}\label{eq:fnforintegral}
  f_n(x) = \frac{n}{2}\sum_{i=0}^{n}\frac{(-1)^{n-i}}{i!(n-i)!}\left( x-\frac{n}{2}+i\right)^{n-2}| x-\frac{n}{2}+i|,\quad n\geq 2.
\end{equation}
The advantageous tool in  that calculation was the symmetry of the elements of the sequence, with the initial element being the characteristic function of the  symmetric unit interval $(-1/2, 1/2)$.
The current work is  devoted to the calculation of explicit  formula for the recursive sequence of functions given by integro-difference equations. The similarity between the previous work and the current is  the choice of the same first element of the sequence, that is $ f_1(x)  = \chi_{(-1/2, 1/2)} (x)$, and the integral operator involved in the recursion. However the inclusion of the difference operators in the current  problem  adds  significantly to the difficulty of the problem. The difference operator which has no smoothing property is coupled with an integral operator which has a smoothing property, yields a sequence $ g_n $ that has same order of smoothness as the first element  $g_1$ of the sequence. The first element $ g_1(x)= = \chi_{(-1/2, 1/2)} (x) $  has jump discontinuities at the points $ x= -1 $ and $ x = 1 $.
 \section{ Explicit formula for sequences of functions recursively-defined by integrodifference equations }

 The main task here in  this section is to derive the closed-form expression of the sequence of functions define  recursively with specified integro-difference relation:
\begin{equation}\label{eq:integrodifference}
  g_1(x)  = \chi_{(-1/2, 1/2)} (x),\quad g_{n+1}(x)  = g_n(x + 1/2)- g_n(x- 1/2) + \int_{x-\frac{1}{2}}^{x + \frac{1}{2}}  g_n(s)ds, \, n\in \mathbb{N}.
\end{equation}
 Let us define the shift operators $E^{\frac{1}{2}} $  and $ E^{- \frac{1}{2}}$ as follows:
\begin{equation}\label{leftandrightshifts}
    E^ {\frac{1}{2}}u(x):= u(x + 1/2), \quad  E^ {-\frac{1}{2}} u(x):= u(x-1/2).
\end{equation}
Let us denote  by $ L $, the operator which is the difference of the left and the right shift operators

\begin{equation}\label{eq:theoperatorL}
    L u(x):= \left( E^ {\frac{1}{2}} - E^ {-\frac{1}{2}}\right)u(x)= u(x +1/2) -u(x- 1/2).
 \end{equation}
 Let us denote by $ K $, the integral operator

 \begin{equation}\label{eq:theoperatorK}
     Ku(x):= \int_{x-\frac{1}{2}}^{x+\frac{1}{2}}u(s)ds.
 \end{equation}
 The integral operators $ K $ and the difference operator $ L $ are related as follows
 \begin{equation}\label{eq:KandLrelated}
     \frac{d}{dx} Ku(x):=  \frac{d}{dx}\int_{x-\frac{1}{2}}^{x+\frac{1}{2}}u(s)ds = u(x +1/2) -u(x-1/2)= Lu(x).
 \end{equation}
Using (\ref{eq:theoperatorL}) and (\ref{eq:theoperatorK}), the integrodifference recurrence relation (\ref{eq:integrodifference})
may be written as
\begin{equation}\label{eq:recurrencewithL}
   g_1(x)  = \chi_{(-1/2, 1/2)} (x),\quad g_{n+1}(x)  = L g_n(x) +  K g_n(x), \, n\in \mathbb{N},
\end{equation}
 and (\ref{eq:integralrecursion}) may be written as
 \begin{equation}\label{eq:recurrencewithK}
   f_1(x)  = \chi_{(-1/2, 1/2)} (x),\quad f_{n+1}(x)  =   K f_n(x), \, n\in \mathbb{N}.
\end{equation}

\begin{lemma}
 The integral operator $ K $ and the difference operator $ L $  commute.
 \begin{equation}\label{eq:commutative}
   L Ku(x)= L \int_{x-\frac{1}{2}}^{x + \frac{1}{2}} u(s)ds =\int_{x-\frac{1}{2}}^{x + \frac{1}{2}}L u(s)ds= KLu(x).
 \end{equation}
\end{lemma}
\begin{proof}
  By using the operational definition of $E^ {\frac{1}{2}}$ and $E^ {-\frac{1}{2}}$ that are given in (\ref{leftandrightshifts}),
  \begin{align}\label{proofofcommutative}
   LKu(x)=  L \int_{x-\frac{1}{2}}^{x + \frac{1}{2}} u(s)ds &= E^ {\frac{1}{2}} \int_{x-\frac{1}{2}}^{x + \frac{1}{2}}  u(s)ds - E^ {-\frac{1}{2}}\int_{x-\frac{1}{2}}^{x + \frac{1}{2}}  u(s)ds \nonumber \\
     & =  \int_{x}^{x+1} u(s)ds - \int_{x-1}^{x}u(s)ds  \nonumber \\
    & =   \int_{x-\frac{1}{2}}^{x+\frac{1}{2}}  u(s+1/2)ds -  \int_{x-\frac{1}{2}}^{x+\frac{1}{2}}  u(s-1/2)ds \nonumber  \\
     & = \int_{x-\frac{1}{2}}^{x+\frac{1}{2}} [ u(s+1/2) -u(s-1/2)]ds  \nonumber \\
     & = \int_{x-\frac{1}{2}}^{x+\frac{1}{2}} [E^ {\frac{1}{2}} - E^ {-\frac{1}{2}}] u(s) ds= K Lu(x).
  \end{align}
   Hence the Lemma is proved.
\end{proof}
Now we list down the first few elements of the sequence $g_n $. Let us see the connection between the sequence $g_n $ given in (\ref{eq:integrodifference}) and the sequence $ f_n $ given in (\ref{eq:fnforintegral}).
\begin{equation}\label{eq:g1eqalsf1}
  g_1(x):= f_1(x)= \chi_{(-1/2, 1/2)} (x).
\end{equation}
\begin{equation}\label{eq:g2relatedtof2}
 g_2(x):= L g_1(x) + Kg_1(x) =  L f_1(x) + Kf_1(x)=  Lf_1(x)+ f_2(x).
\end{equation}
Now by (\ref{eq:recurrencewithL}), (\ref{eq:g2relatedtof2}) and (\ref{eq:g1eqalsf1})
\begin{align}\label{eq:g3relatedtof3}
   g_3(x)&= L g_2(x) + K  g_2 (x)\nonumber \\
   & =L[ L f_1(x)+ f_2(x)] + K[L f_1(x)+ f_2(x)] \nonumber \\
   & = L^2 f_1(x)+ Lf_2(x) +  KLf_1(x) + Kf_2(x)  \nonumber \\
   & = L^2 f_1(x)+ Lf_2(x) +  L K f_1(x) + f_3(x) \nonumber \\
   & = L^2 f_1(x)+ Lf_2(x) +  Lf_2(x) + f_3(x)  \nonumber \\
   &= L^2 f_1(x)+ 2Lf_2(x)  + f_3(x).
   \end{align}
   From these observation  we shall set a relation between the sequences $ f_n $ and $g_n$ given in the following theorem.
   \begin{theorem}
   For appropriate numerical coefficients $ c_{n,i}$, the sequence $ f_n $ given in (\ref{eq:integralrecursion}) and the sequence $ g_n $ given in (\ref{eq:integrodifference}) are related as follows
     \begin{equation}\label{eq:gnrelatedtofn}
       g_n(x)=\sum_{k=0}^{n-1} c_{n,k} L^k f_{n-k}(x).
     \end{equation}
   \end{theorem}
   \begin{proof}
     We use induction over $ n $. The relation given in equation (\ref{eq:gnrelatedtofn}) is valid for $n=1,2,3 $, as was shown in equations (\ref{eq:g1eqalsf1}), (\ref{eq:g2relatedtof2}) and (\ref{eq:g3relatedtof3}). Suppose that the relation given in (\ref{eq:gnrelatedtofn}) is valid for some $ n \in \mathbb{N }$. Then by  the induction assumption, (\ref{eq:recurrencewithL}), (\ref{eq:recurrencewithK}),  and (\ref{eq:commutative})
     \begin{align}\label{eq:provedbyinduction}
       g_{n+1}(x) & = Lg_n(x) + K g_n(x) \nonumber  \\
        & = L \sum_{k=0}^{n-1} c_{n,k}  L^k f_{n-k}(x)+ K \sum_{k=0}^{n-1} c_{n,k} L^k f_{n-k}(x)  \nonumber \\
        & = \sum_{k=0}^{n-1} c_{n,k} L^{k+1} f_{n-k}(x)+   \sum_{k=0}^{n-1} c_{n,k}K L^k  f_{n-k}(x)  \nonumber \\
        & = \sum_{k=0}^{n-1} c_{n,k} L^{k+1} f_{n-k}(x)+   \sum_{k=0}^{n-1} c_{n,k} L^k K f_{n-k}(x)  \nonumber \\
        & = \sum_{k=0}^{n-1} c_{n,k}  L^{k+1} f_{n-k}(x)+ \sum_{k=0}^{n-1} c_{n,k}  L^k f_{n+1-k}(x)  \nonumber \\
        & =  \sum_{k = 0}^{n}  c_{n+1,k} L^k f_{n+1-k} ,
     \end{align}
where
\begin{equation}\label{eq:cofficentsforgk}
  \begin{cases}
c_{n+1,0} =  c_{n,0}= 1,\\
c_{n+1,k} =  c_{n,k} +  c_{n,k-1},\quad k = 1,2,...,n-1, \\
c_{n+1,n} =  c_{n,n-1}= 1.
\end{cases}
\end{equation}
This is the recurrence relation for the binomial coefficients which are the solutions of the recurrence relation (\ref{eq:cofficentsforgk}).
Therefore
\begin{equation}\label{eq:coefficentsdetermined}
  c_{n,k}= \binom{n-1}{k}.
\end{equation}
Therefore
\begin{equation}\label{eq:gn}
g_n(x)=\sum_{k=0}^{n-1} \binom{n-1}{k} L^k f_{n-k}(x).
  \end{equation}
  Hence the theorem is proved.
   \end{proof}
\begin{theorem}
  The solution of the sequence (\ref{eq:integrodifference}) is given  in terms of  the elements of the sequence $f_n$ by
  \begin{equation}\label{eq:gnintermsoffn}
     g_n(x)= \sum_{k=0}^{n-1}\binom{n-1}{k}\sum_{r=0}^{k} (-1)^r\binom{k}{r}f_{n-k}(x+\frac{k}{2}-r).
  \end{equation}
  \end{theorem}
\begin{proof}
  The binomial expansion of  $ L^k = \left( E ^ \frac{1}{2}- E ^ {-\frac{1}{2}}\right)^k $ that appear in (\ref{eq:gn}) yields
   \begin{equation}\label{eq:expansionofLtothek}
       L^k = \left( E ^ \frac{1}{2}- E ^ {-\frac{1}{2}}\right)^k =\sum_{r=0}^{k} (-1)^r\binom{k}{r}\left(E ^ {-\frac{1}{2}}\right)^r \left(E ^ {\frac{1}{2}}\right)^{k-r} = \sum_{r=0}^{k} (-1)^r \binom{k}{r} E ^{\frac{k}{2}-r}.
   \end{equation}
Replacing (\ref{eq:expansionofLtothek}) into (\ref{eq:gn}) we get
\begin{align*}
     g_n(x) & =\sum_{k=0}^{n-1}\binom{n-1}{k}\sum_{r=0}^{k} (-1)^r\binom{k}{r}E ^{\frac{k}{2}-r}f_{n-k}(x) \nonumber \\
      & = \sum_{k=0}^{n-1}\binom{n-1}{k}\sum_{r=0}^{k} (-1)^r\binom{k}{r}f_{n-k}(x+\frac{k}{2}-r ).
\end{align*}
Thus the theorem is proved.
\end{proof}

 At this point it is desirable to write the closed form solution of $g_n$ in terms of the closed form solution of $f_n$ without $ f_n$ being involved as in the case of (\ref{eq:gnintermsoffn}).  In (\ref{eq:gnintermsoffn}), for each $n \in \mathbb{N}$ the term $f_{n-k}$ equals $f_1 $ whenever $k=n-1$. The closed form  solution for the recursive sequence $ f_n$  given in (\ref{eq:fnforintegral}) is defined  for $n \geq 2 $, and $ f_1(x)= \chi_{(-1/2, 1/2)}(x)$ can not be derived from (\ref{eq:fnforintegral}). In deed, if we put $ n=1 $ in the right hand side of the equation (\ref{eq:fnforintegral}) we get

$$ \frac{1}{2}\sum_{i=0}^{1}\frac{(-1)^{1-i}}{i!(1-i)!}\frac{|x-1/2+i|}{x-1/2+i}= \frac{1}{2}\left(\frac{|x+1/2|}{x+1/2}-\frac{|x-1/2|}{x-1/2} \right)= \chi_{(-1/2, 1/2)}(x), \quad x \neq \pm 1/2, $$
 whereas $f_1(x) := \chi_{(-1/2, 1/2)}(x) $ is defined in $ \mathbb{R} $. In particular, $f_1(\pm 1/2)=0 $.
%

 \begin{theorem}
  For $ n \geq 2 $, the explicit solution of the sequence (\ref{eq:integrodifference}) is given by
  \begin{align*}
    g_n(x)&= \sum_{k=0}^{n-2} \sum_{r=0}^{k} \sum_{t=0}^{n-k} \binom{n-1}{k}  \binom{k}{r} \frac{n-k}{2}\frac{(-1)^{n-k-t}}{t!(n-k-t)!}\left(  x + k-r-\frac{n}{2}+t\right)^{n-k-2}| x + k-r-\frac{n}{2}+t| \\
     & + \sum_{r=0}^{n-1}\binom{n-1}{r}(-1)^r g_1(x-\frac{n-1}{2}-r).
  \end{align*}
  \end{theorem}
 %
\begin{proof}
From (\ref{eq:gnintermsoffn}) we have
\begin{equation}\label{eq:indexseparated}
   g_n(x)= \sum_{k=0}^{n-2}\binom{n-1}{k}\sum_{r=0}^{k} (-1)^r\binom{k}{r}f_{n-k}(x+\frac{k}{2}-r)  +  \sum_{r=0}^{n-1}(-1)^r \binom{n-1}{r}f_{1}(x +\frac{n-1}{2}-r).
\end{equation}
   According to (\ref{eq:fnforintegral}), for $ n\geq 2 $
   \begin{equation}\label{eq:fnminusk}
     f_{n-k}(x)= \frac{n-k}{2}\sum_{t=0}^{n-k}\frac{(-1)^{n-k-t}}{t!(n-k-t)!}\left( x-\frac{n-k}{2}+ t\right)^{n-k-2}| x-\frac{n-k}{2}+t|,\quad 0 \leq k \leq n-2.
   \end{equation}
Therefore shifting $ f_{n-k}$  that is written in (\ref{eq:fnminusk})  by $ \frac{k}{2}-r $ units yields
 \begin{equation}\label{eq:fnminuskshifted}
     f_{n-k}(x+\frac{k}{2}-r)= \frac{n-k}{2}\sum_{t=0}^{n-k}\frac{(-1)^{n-k-t}}{t!(n-k-t)!}\left( x +k-r-\frac{n}{2}+ t \right)^{n-k-2}| x + k-r-\frac{n}{2}+t|,\quad  0 \leq k \leq n-2.
   \end{equation}
Now plugging the expression of $f_{n-k}(x+\frac{k}{2}-r)$ that is written (\ref{eq:fnminuskshifted}) into (\ref{eq:indexseparated}), we get a complete closed-form solution of the recursive integrodifference sequence $(g_n)$ that is defined in (\ref{eq:integrodifference}) as:

   \begin{align}\label{eq:gncompleted}
     g_{n}(x) &= \sum_{k=0}^{n-2} \binom{n-1}{k} \sum_{r=0}^{k} \binom{k}{r} \frac{n-k}{2}\sum_{t=0}^{n-k}\frac{(-1)^{n-k-t}}{t!(n-k-t)!}\left(  x + k-r-\frac{n}{2}+t\right)^{n-k-2}| x + k-r-\frac{n}{2}+t| \nonumber\\
     & + \sum_{r=0}^{n-1}(-1)^r \binom{n-1}{r}f_{1}(x +\frac{n-1}{2}-r) \nonumber \\
      &= \sum_{k=0}^{n-2} \sum_{r=0}^{k} \sum_{t=0}^{n-k} \binom{n-1}{k}  \binom{k}{r} \frac{n-k}{2}\frac{(-1)^{n-k-t}}{t!(n-k-t)!}\left(  x + k-r-\frac{n}{2}+t\right)^{n-k-2}| x + k-r-\frac{n}{2}+t| \nonumber \\
      & + \sum_{r=0}^{n-1}(-1)^r \binom{n-1}{r}g_{1}(x +\frac{n-1}{2}-r).
   \end{align}
This completes the proof.
  \end{proof}

  \section*{Conclusions}
  In this paper we have computed the closed form expression  for a sequence of functions defined via recursive integro-difference relation. The contrasting properties of differences (way similar to differentials) and integrals make the calculations harder. However  the considerations of symmetry of the problem and the authors use of his perviously published results make the task manageable. The real life applications of such problems were mentioned. The author believes that problem can be diversified and upgraded with  further applications. Interested researcher can consider a similar problem with a different initial element $g_1$.

\section*{Conflict of interests}
The author declare that there is no conflict of interests regarding the publication of this paper.

\section*{Funding}
This research work is not funded by any organization or individual/s.

\section*{Acknowledgment}
The author is thankful to the anonymous reviewers for their constructive and valuable suggestions.

\section*{Data availability}
Data sharing is not applicable to this article as no data were collected or analysed during the current study.

\end{document}